\documentclass{article}

\usepackage{float}

\usepackage[square,numbers]{natbib}

\usepackage[utf8]{inputenc} 
\usepackage[T1]{fontenc}    
\usepackage{hyperref}       
\usepackage{url}            
\usepackage{booktabs}       
\usepackage{amsthm, amsmath, amsfonts, amssymb}  
\usepackage{nicefrac}       
\usepackage{microtype}      
\usepackage{graphicx}
\usepackage{tabularx}
\usepackage{hyperref}
\usepackage{multirow}
\usepackage{extarrows}
\usepackage[letterpaper,margin=0.9in]{geometry}

\usepackage{color}
\usepackage[all]{xy}
\usepackage{tikz}
\usepackage{tikz-cd}
\usepackage{caption}
\usepackage{subcaption}

\newtheorem{thm}{Theorem}[section]
\newtheorem{lem}[thm]{Lemma}
\newtheorem{coro}[thm]{Corollary}

{\theoremstyle{definition}\newtheorem{exam}[thm]{Example}}

{\theoremstyle{definition}\newtheorem{rem}[thm]{Remark}}

\title{Homotopy Types Of Toric Orbifolds From Weyl Polytopes} 

\author{
  Tao Gong\\
  \texttt{tgong23@uwo.ca} \\
}

\date{}

\begin{document}
\maketitle

\begin{abstract}
 Given a reduced crystallographic root system with a fixed simple system, it is associated to a Weyl group $W$, parabolic subgroups $W_K$'s and a polytope $P$ which is the convex hull of a dominant weight.
 The quotient $P/W_K$ can be identified with a polytope. Polytopes $P$ and $P/W_K$ are associated to toric varieties $X_P$ and $X_{P/W_K}$ respectively. It turns out the underlying topological spaces $X_P/W_K$ and $X_{P/W_K}$ are homotopy equivalent, when considering the polytopes in the real span of the root lattice or of the weight lattice. 
\end{abstract}

\section{Introduction}

Let $R$ be a reduced crystallographic root system in a real Euclidean space $V$ with a fixed system of simple roots, those hyperplanes orthogonal to roots divide $V$ into connected components, called chambers. 
There is an associated Weyl group $W$ to $R$ that is generated by
reflections by the hyperplanes. The convex hull $P$ of the $W$-orbit of a  dominant weight is a polytope, called a Weyl polytope. The polytope $P$ carries a natural action of $W$, and the quotient space $P/W$ can be identified with another polytope contained in $P$, which is the intersection of $P$ and a fixed closed chamber $\bar{C}$ of the $W$-action.
Then we can consider the associated normal fans $\Sigma_P$ and $\Sigma_{P/W}$ in the space $V$, regarded as the real span of the coweight lattice. Associated to normal fans, there are toric varieties $X_P$ and $X_{P/W}$, and $X_P$ carries a $W$-action which comes from that over $P$. 

Here is a natural question:  is $X_P/W$ isomorphic to $X_{P/W}$?
\cite{blume2015toric}  showed positive answers with respect to varieties  for Lie types $A$, $B$ and $C$.  This paper gives out a positive answer for an arbitrary root system $R$ in the topological level.

\begin{thm}
 There is a homotopy equivalence $\Phi:X_P/W \to X_{P/W}$.
\end{thm}

This result is a special case of a general one. For any parabolic subgroup $W_K$ of $W$  associated to a subset of simple roots, the quotient space $P/W_K$ can be identified with a polytope,  which is the intersection of $P$ and a fixed fundamental domain $\overline{C_K}$ of the $W_K$-action on $V$. The normal fan $\Sigma_{P/W_K}$ of $P/W_K$, considered in the real span of the coweight lattice, is associated to a toric variety $X_{P/W_K}$.
\begin{thm}
  There is a homotopy equivalence $\Phi_K:X_P/W_K \to X_{P/W_K}$ for any parabolic subgroup $W_K$ of $W$.
\end{thm}

For the proof of the result, we will consider the topological models of the varieties $X_{P}$ and $X_{P/W_K}$, then show that $X_P/W_K$ and $X_{P/W_K}$ are finite simplicial complexes and the fiber of $\Phi_K$ is (locally) contractible, which will lead to the homotopy equivalence. It is worth mentioning that the fiber of $\Phi_K$ is, roughly speaking, homeomorphic to the quotient of a torus under the action of its corresponding Weyl group.   One can see the Figure \ref{space of A1} below for spaces assocaited to the root system of type $A_1$.

\begin{figure}[h!]
  \caption{Associated topological spaces to Lie type $A_1$}\label{space of A1}
  \captionsetup[subfigure]{font=footnotesize}
  \centering
  \subcaptionbox*{$X_P\cong S^2$}[.3\textwidth]{%
   \tikzset{every picture/.style={line width=0.75pt}} 
  \begin{tikzpicture}[x=0.75pt,y=0.75pt,yscale=-1,xscale=1]
\draw   (35.85,96.56) .. controls (35.85,68.95) and (58.23,46.56) .. (85.85,46.56) .. controls (113.46,46.56) and (135.85,68.95) .. (135.85,96.56) .. controls (135.85,124.18) and (113.46,146.56) .. (85.85,146.56) .. controls (58.23,146.56) and (35.85,124.18) .. (35.85,96.56) -- cycle ;
\draw   (85.53,46.5) .. controls (93.4,46.47) and (99.84,68.85) .. (99.93,96.48) .. controls (100.02,124.12) and (93.71,146.54) .. (85.85,146.56) .. controls (77.98,146.59) and (71.53,124.21) .. (71.45,96.57) .. controls (71.36,68.94) and (77.67,46.52) .. (85.53,46.5) -- cycle ;
\end{tikzpicture}
  }%
\subcaptionbox*{$X_P/W\cong$ quotient of a hemisphere by collapsing the equator to a segment}[.3\textwidth]{
 \tikzset{every picture/.style={line width=0.75pt}} 
\begin{tikzpicture}[x=0.75pt,y=0.75pt,yscale=-1,xscale=1]
\draw  [draw opacity=0] (266,46.56) .. controls (293.54,46.65) and (315.85,69) .. (315.85,96.56) .. controls (315.85,124.12) and (293.54,146.48) .. (266,146.56) -- (265.85,96.56) -- cycle ; \draw   (266,46.56) .. controls (293.54,46.65) and (315.85,69) .. (315.85,96.56) .. controls (315.85,124.12) and (293.54,146.48) .. (266,146.56) ;  
\draw   (265.85,46.56) .. controls (273.71,46.54) and (280.16,68.92) .. (280.25,96.55) .. controls (280.33,124.18) and (274.03,146.6) .. (266.16,146.63) .. controls (258.29,146.65) and (251.85,124.27) .. (251.76,96.64) .. controls (251.67,69.01) and (257.98,46.59) .. (265.85,46.56) -- cycle ;
\draw [color={rgb, 255:red, 208; green, 2; blue, 27 }  ,draw opacity=1 ]   (254,67) -- (276,67.55) ;
\draw [shift={(278,67.6)}, rotate = 181.43] [color={rgb, 255:red, 208; green, 2; blue, 27 }  ,draw opacity=1 ][line width=0.75]    (10.93,-3.29) .. controls (6.95,-1.4) and (3.31,-0.3) .. (0,0) .. controls (3.31,0.3) and (6.95,1.4) .. (10.93,3.29)   ;
\draw [color={rgb, 255:red, 208; green, 2; blue, 27 }  ,draw opacity=1 ]   (251.76,96.64) -- (278.25,96.56) ;
\draw [shift={(280.25,96.55)}, rotate = 179.82] [color={rgb, 255:red, 208; green, 2; blue, 27 }  ,draw opacity=1 ][line width=0.75]    (10.93,-3.29) .. controls (6.95,-1.4) and (3.31,-0.3) .. (0,0) .. controls (3.31,0.3) and (6.95,1.4) .. (10.93,3.29)   ;
\draw [color={rgb, 255:red, 208; green, 2; blue, 27 }  ,draw opacity=1 ]   (253.85,126.26) -- (275.85,126.81) ;
\draw [shift={(277.85,126.86)}, rotate = 181.43] [color={rgb, 255:red, 208; green, 2; blue, 27 }  ,draw opacity=1 ][line width=0.75]    (10.93,-3.29) .. controls (6.95,-1.4) and (3.31,-0.3) .. (0,0) .. controls (3.31,0.3) and (6.95,1.4) .. (10.93,3.29)   ;
\end{tikzpicture}}
\subcaptionbox*{$X_{P/W}\cong S^2$}[.3\textwidth]{
\tikzset{every picture/.style={line width=0.75pt}} 

\begin{tikzpicture}[x=0.75pt,y=0.75pt,yscale=-1,xscale=1]

\draw   (420.85,106.56) .. controls (420.85,92.75) and (432.04,81.56) .. (445.85,81.56) .. controls (459.65,81.56) and (470.85,92.75) .. (470.85,106.56) .. controls (470.85,120.37) and (459.65,131.56) .. (445.85,131.56) .. controls (432.04,131.56) and (420.85,120.37) .. (420.85,106.56) -- cycle ;
\draw   (445.69,81.5) .. controls (450.11,81.49) and (453.72,92.68) .. (453.77,106.51) .. controls (453.81,120.33) and (450.26,131.55) .. (445.85,131.56) .. controls (441.43,131.58) and (437.81,120.38) .. (437.77,106.56) .. controls (437.72,92.73) and (441.27,81.52) .. (445.69,81.5) -- cycle ;

\end{tikzpicture}
  }
\end{figure}
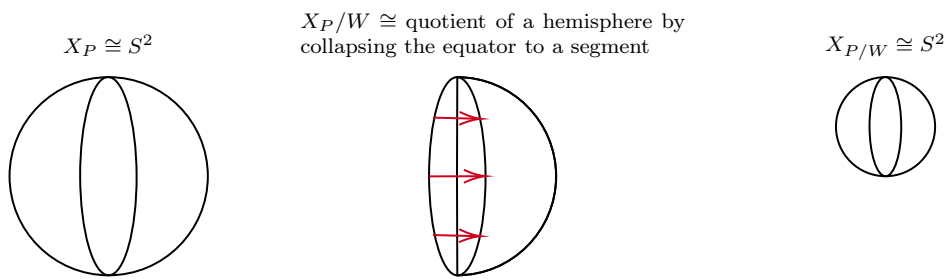

\begin{coro}
  $\Phi_K$ induces the isomorphism of (co)homology of $X_P/W_K$ and $X_{P/W_K}$ with any coefficient ring.
\end{coro}
This corollary covers a result from \cite{horiguchi2021toric}, which showed the rational cohomology isomorphism of $X_P/W_K$ and $X_{P/W_K}$ for Lie types $A$, $B$, $C$ and $D$.

We will get similar results in the dual setting. Consider the Weyl polytope in the real span of weight lattice, which is still $V$ but with a different generating lattice, for this reason we call this polytope an alternative Weyl polytope, denoted by $Q$. Then the normal fans $\Sigma_{Q}, \Sigma_{Q/W_K}$ are lying in the real span of the coroot lattice, and they are associated to toric varieties $X_{Q}, X_{Q/W_K}$ respectively.

\begin{thm}
  There is a homotopy equivalence $\Psi_K:X_{Q}/W_K \to X_{Q/W_K}$ for any parabolic subgroup $W_K$ of $W$.
 \end{thm}

In Section 2, we will review how to construct a toric variety from a polytope, and how to obtain a topological model of the toric variety.
In Section 3, we will review root systems and show the structure of the Weyl polytope and the quotient polytope.
In Section 4, we will prove the special case of the main result with respect  to the Weyl group from the Weyl polytope.
In Section 5, we will prove the general form of the main result with respect to the parabolic subgroup of the Weyl group from the Weyl polytope.
In Section 6, we will prove the dual version of the main result from the alternative Weyl polytope.

\subsubsection*{Acknowledgments. {\rm Thanks to Matthias Franz for inspiring discussions, also to Sayantan Roy Chowdhury on Algebraic Geometry.}}

\section{Toric varities from polytopes}
We mainly follow \cite{cox2011toric} and \cite{fulton1993introduction} for  toric varieties. 

Given a lattice $M$ of rank $n$, its dual lattice $N$ and a lattice polytope $P$ of dimension $n$ in $M_{\mathbb{R}}:=M\otimes_{\mathbb{Z}}\mathbb{R}$, each facet $F$ of $P$ admits a normal vector $l_{F}$ in $N_{\mathbb{R}}:=N\otimes_{\mathbb{Z}}\mathbb{R}$, which can be chosen integral, primitive and pointing inside the polytope $P$. Define the \textbf{normal fan} $\Sigma_P$ whose cones are generated by those sets of normal vectors $l_{F_{i_1}},l_{F_{i_2}},\cdots,l_{F_{i_k}}$ whose corresponding facets $F_{i_1},F_{i_2},\cdots,F_{i_k}$ have non-empty intersection in $P$.
Each cone $\sigma$ in  $\Sigma_P$ corresponds to an affine toric variety $X_{\sigma}$, and these affine varieties fit together to form an algebraic variety, denoted by $X_P$.

More generally, there is a bijection between rational fans in $\mathbb{R}^n$ and toric varieties of complex dimension $n$, see \cite[Theorem 3.1.19]{cox2011toric}. A toric variety is compact if and only if the corresponding fan is \textbf{complete}; a toric variety is smooth if and only if the corresponding fan is \textbf{smooth}; a toric variety is a \textbf{toric orbifold} (that is, it is locally isomorphic to a quotient of $\mathbb{C}^n$ by a finite group action) if and only if the corresponding fan is \textbf{simplicial}.

For any cone $\sigma$ in $\Sigma_P$, let $\sigma^{\vee}$ denote $\{v\in M_\mathbb{R}:\,v(u)\ge 0,\forall u\in \sigma\}$, $\sigma^{\bot}$ denote $\{v\in M_\mathbb{R}:\,v(u)= 0,\forall u\in \sigma\}$.
The points of $X_{\sigma}$ can be described as the semigroup homomorphisms $$\mathrm{Hom}_{sg}(\sigma^{\vee}\cap M, \mathbb{C}),$$ where $\mathbb{C}$ is the multiplicative semigroup $\mathbb{C}^{*}$ with $\{0\}$. 
Let $\mathbb{R}_{\ge}$ denote $\mathbb{R}^{+}\cup\{0\}$, the multiplicative semigroup of non-negative real numbers. In this case there is a retraction given by the absolute value map:
$$\mathbb{R}_{\ge}\hookrightarrow \mathbb{C}\xrightarrow{|\cdot|}\mathbb{R}_{\ge}.$$
This determines a subspace
$$(X_{\sigma})_{\ge}:=\mathrm{Hom}_{sg}(\sigma^{\vee}\cap M, \mathbb{R}_{\ge})\subset X_{\sigma}=\mathrm{Hom}_{sg}(\sigma^{\vee}\cap M, \mathbb{C})$$
together with a retraction $X_{\sigma}\to (X_{\sigma})_{\ge}$.
Actually, when one embeds $X_\sigma$ into some $\mathbb{C}^l$ as the zero set of several polynomials, $(X_{\sigma})_{\ge}$ is identified with $X_{\sigma}\cap (\mathbb{R}_{\ge})^l\subset \mathbb{C}^l$. Thus $(X_{\sigma})_{\ge}$ is a closed subspace of $X_{\sigma}$. All $(X_{\sigma})_{\ge}$'s fit together to form a closed subspace $(X_P)_{\ge}$ of $X_P$  with a retraction
$$(X_P)_{\ge}\hookrightarrow X_P\to (X_P)_{\ge}.$$

\begin{exam}
  Let $P$ be the interval $[-1,1]$ in $\mathbb{R}$. Then $\Sigma_P=\left\{\{0\},\,[0,\infty),\,(-\infty,0]\right\}$, $X_{\{0\}}=\mathbb{C}^*$, $X_{[0,\infty)}=\mathbb{C}$ and $X_{(-\infty,0]}=\mathbb{C}$. There are two pushout diagrams of topological spaces
  $$\begin{tikzcd}
    \mathbb{C}^*\arrow[r,"x\mapsto x^{-1}"]\arrow[d,"x\mapsto x"']\arrow[rd,phantom,"\ulcorner" very near end] &  \mathbb{C}\arrow[d]\\
    \mathbb{C}\arrow[r] & X_{P},
  \end{tikzcd}\quad\qquad
  \begin{tikzcd}
    \mathbb{R}^+\arrow[r,"x\mapsto x^{-1}"]\arrow[d,"x\mapsto x"']\arrow[rd,phantom,"\ulcorner" very near end] &  \mathbb{R}_{\ge}\arrow[d]\\
    \mathbb{R}_{\ge}\arrow[r] & (X_{P})_{\ge}.
  \end{tikzcd}$$
  Thus $X_P\cong \mathbb{CP}^1\cong S^2$, $(X_{P})_{\ge}$ is a meridian in $S^2$.
\end{exam}

Consider a toric orbit $O(\sigma)=\mathrm{Hom}_{sg}(\sigma^{\bot}\cap M,\mathbb{C}^*)$ in $X_P$ and its non-negative points $O(\sigma)_{\ge}:=O(\sigma)\cap (X_P)_{\ge}=\mathrm{Hom}_{sg}(\sigma^{\bot}\cap M,\mathbb{R}^+)\subset (X_P)_{\ge}$. 
The canonical map
\begin{align*}
  \mathrm{Hom}_{sg}(M,S^1)\times O(\sigma)_{\ge}&\to  O(\sigma)\\
 (f,g)&\mapsto \left(x\mapsto f(x)\cdot g(x)\right)
\end{align*} 
gives out an isomorphism induced by $\sigma^{\bot}\cap M\hookrightarrow M$ as below:
\begin{align*}
  \mathrm{Hom}_{sg}(\sigma^{\bot}\cap M,S^1)\times O(\sigma)_{\ge}\xrightarrow{\cong}  O(\sigma),
\end{align*}
since there is an isomorphism of semigroups $S^1\times \mathbb{R}^+\cong\mathbb{C}^*$.

Let $N({\sigma})$ be the sublattice of $N$ generated by $\sigma\cap N$, $S_N$ be the compact torus $N_{\mathbb{R}}/N=N\otimes_{\mathbb{Z}}S^1$, $S_{N(\sigma)}$ be the subtorus of $S_N$ defined by $\frac{N(\sigma)\otimes_{\mathbb{Z}}\mathbb{R}+N}{N}$. 
Then $\mathrm{Hom}_{sg}(\sigma^{\bot}\cap M,S^1)=N/N({\sigma})\otimes_{\mathbb{Z}}S^1$ and the diagram 
$$\begin{tikzcd}
  S_N\times O(\sigma)_{\ge}\arrow[rr]\arrow[rd] & & O(\sigma)\\
  & \frac{S_N}{S_{N(\sigma)}}\times O(\sigma)_{\ge}\arrow[ru,"\cong"'] &
\end{tikzcd}$$
commutes. Put all toric orbits $O(\sigma)$'s together, we get a map
$$S_N\times (X_P)_{\ge}\to X_P$$
which realizes $X_P$ as a quotient space.

Recall that for any integer $z\ge n$, the multiple $zP$ is again a lattice polytope of dimension $n$. Polytopes $P$ and $zP$ are homeomorphic via the multiplication by $z$, and they have the same inward-pointing facet normals, hence $\Sigma_P=\Sigma_{zP}$, $X_P=X_{zP}$. Assume that $(zP)\cap M=\{m_0,m_1,\cdots,m_k\}$, then there is 
an embedding
\begin{align*}
  \rho:  X_P&\hookrightarrow \mathbb{CP}^k=\frac{\mathbb{C}^{k+1}\setminus\{0\}}{\mathbb{C}^*}\\
   x &  \mapsto  \left[\chi^{m_0}(x),\chi^{m_1}(x),\cdots,\chi^{m_k}(x)\right].
\end{align*}
More precisely, when $x$ lies in $X_\sigma=\mathrm{Hom}_{sg}(\sigma^{\vee}\cap M, \mathbb{C})$ with $\sigma$ corresponding to some vertex $m_i$ of $zP$, then for any $m\in (zP)\cap M$, $\chi^{m}(x)$ is defined to be $x(m-m_i)$ since $m-m_i\in\sigma^{\vee}\cap M$. See \cite[Theorem 2.3.1]{cox2011toric} for this embedding $\rho$.
From $\rho$, we can get the \textbf{moment map} $\mu: X_P \to M_{\mathbb{R}}$ given by
\begin{align*}
     \mu(x)= \frac{1}{\sum\limits_{m\in (zP)\cap M}|\chi^m(x)|}\sum\limits_{m\in (zP)\cap M}|\chi^m(x)|m.
\end{align*}
Its restriction to $(X_P)_{\ge}$, denoted by $\tilde{\mu}$,  gives out an homeomorphism onto $zP$, for which one can refer to 
\cite[$\S$4.2]{fulton1993introduction} or \cite[Theorem 12.2.5]{cox2011toric}. Furthermore, if a face $F$ of $P$ correponds to the cone $\sigma$ in $\Sigma_P$, then $\tilde{\mu}$ maps $O(\sigma)_{\ge}$ bijectively to the relative interior of $zF$.

Due to all the above, we get a topological model for $X_P$. 
There is an $S_N$-equivariant homeomorphism 
\begin{align}
  X_P\cong \left(S_N\times P\right)/\sim,
\end{align}
where $(t_1,p_1)\sim (t_2,p_2)$ if and only if $p_1=p_2$ and
$t_1$, $t_2$ are congruent modulo the subtorus $S_{N({\sigma_{p_1}})}$ of $S_N$ for the cone $\sigma_{p_1}\in\Sigma_P$ corresponding to the face of $P$ whose relative interior contains $p_1$. Write $[t,p]$ for a point of $X_P$.

Now we show that the topological model for $X_P$ is natural with respect to special lattice isomorphisms.  Let $(N',M',P',\Sigma_{P'})$ be another tuple of a rank-$n$ lattice, its dual, a lattice polytope of dimension $n$ and its normal fan.
Suppose that $r:N\to N'$ is an isomorphism of lattices such that $r^*:M'_{\mathbb{R}}\to M_{\mathbb{R}}$ maps $z'P'$ bijectively onto $zP$ for some integers $z,z'\ge n$. 
Then clearly $r$ is \textbf{compatible} with fans, i.e., $r_*:N_\mathbb{R}\to N'_\mathbb{R}$ maps cones to cones.
One can consider the following diagram 
\begin{equation}\label{Naturality}
  \begin{tikzcd}
    S_N\times (X_P)_{\ge}\arrow[r]\arrow[d,"r_*"'] & X_P\arrow[d,"r_*"]\arrow[r,"\mu"] & M_\mathbb{R}\arrow[r,"\frac{1}{z}"] & M_\mathbb{R}\\
    S_{N'}\times (X_{P'})_{\ge}\arrow[r] & X_{P'}\arrow[r,"\mu'"'] & M'_\mathbb{R}\arrow[u,"r^*"']\arrow[r,"\frac{1}{z'}"'] & M'_\mathbb{R}\arrow[u,"\frac{z'}{z}r^*"']
  \end{tikzcd}
\end{equation}
where $r_*$ (resp. $r^*$) is the covariant (resp. contravariant) map induced by $r$.

\begin{lem}\label{Natural mu}
  Diagram \ref{Naturality} is commutative.
\end{lem}
\begin{proof}
  For the left square, suppose that $u\in S_N=\mathrm{Hom}_{sg}(M,S^1)$  and $x\in (X_\sigma)_{\ge}$ for some $\sigma\in\Sigma_P$, and $r_*(\sigma)$ is equal to a cone $\sigma'\in \Sigma_{P'}$, then $(u,x)$ gives out a point $ux$ of $X_\sigma$ by 
  $$ux(m)=u(m)\cdot x(m)$$
  for $m\in \sigma^{\vee}\cap M$; and $r_*(ux)$ gives out a point in $X_{\sigma'}$ by
  $$r_*(ux)(m')=ux(r^*(m'))=u(r^*(m'))\cdot x(r^*(m'))=r_*u(m')\cdot r_*x(m')$$
  for  $m'\in (\sigma')^{\vee}\cap M'$, while the rightmost item is given by $(r_*u,r_*x)$.
  
  For the middle square, suppose that $r_*$ sends a maximal cone $\tau\in\Sigma_P$ to $\tau'\in\Sigma_{P'}$ with corresponding vertices $m_i$ of $zP$ and $m_j'$ of $z'P'$ respectively. Let $y$ be in $X_\tau$, then 
  \begin{align*}
    r^*\mu'r_*(y) &= \frac{1}{\sum\limits_{m'\in (z'P')\cap M'}|\chi^{m'}(r_*(y))|}\sum\limits_{m'\in(z'P')\cap M'}|\chi^{m'}(r_*(y))|r^*(m')\\
    &= \frac{1}{\sum\limits_{m'\in (z'P')\cap M'}|\chi^{m'}(r_*(y))|}\sum\limits_{m'\in(z'P')\cap M'}|r_*(y)(m'-m_j')|r^*(m')\\
    &= \frac{1}{\sum\limits_{m'\in (z'P')\cap M'}|\chi^{m'}(r_*(y))|}\sum\limits_{m'\in(z'P')\cap M'}|y(r^*(m')-m_i)|r^*(m')\\
    & = \frac{1}{\sum\limits_{m'\in (z'P')\cap M'}|\chi^{r^*(m')}(y)|}\sum\limits_{m'\in (z'P')\cap M'}|\chi^{r^*(m')}(y)|r^*(m')\\
    &= \frac{1}{\sum\limits_{m\in (zP)\cap M}|\chi^{m}(y)|}\sum\limits_{m\in (zP)\cap M}|\chi^{m}(y)|m=\mu(y),
  \end{align*}
  where the fifth equality follows from that $(z'P')\cap M'\xrightarrow[\cong]{r^*} (zP)\cap M$.

  The right square square is clearly commutative.
\end{proof}
Consequently, one gets the following diagram
$$
\begin{tikzcd}
  (X_P)_{\ge}\arrow[r,"\tilde{\mu}"]\arrow[d,"r_*"'] & zP\arrow[d,"(r^*)^{-1}"]\arrow[r,"\frac{1}{z}"] & P \arrow[d,"\frac{z}{z'}(r^*)^{-1}"] \\
  (X_{P'})_{\ge}\arrow[r,"\tilde{\mu}'"'] & z'P'\arrow[r,"\frac{1}{z'}"'] & P'
\end{tikzcd}
$$
is commutative. Therefore, we can say the homeomorphism $X_P\cong \left(S_N\times P\right)/\sim$ is natural, in the sense that the following diagram commutes.
\begin{equation}\label{natural model}
  \begin{tikzcd}
    \left(S_N\times P\right)/\sim\arrow[r,"\cong"]\arrow[d,"r_*\times \frac{z}{z'}(r^*)^{-1}"'] & X_P\arrow[d,"r_*"] \\
    \left(S_N\times P'\right)/\sim\arrow[r,"\cong"'] & X_{P'}
  \end{tikzcd}
\end{equation}

\begin{exam}\label{eg:sapces of A1}
  Let $P$ be the interval $[-1,1]$ in $\mathbb{R}$, $P'$ be the interval $[0,1]$. Then 
  \begin{align*}
    X_P=\left(S^1\times [-1,1]\right)/\sim
  \end{align*}
  where $\left(e^{2\pi\sqrt{-1}\theta_1},-1\right)\sim \left(e^{2\pi\sqrt{-1}\theta_2},-1\right)$, $\left(e^{2\pi\sqrt{-1}\theta_1},1\right)\sim \left(e^{2\pi\sqrt{-1}\theta_2},1\right)$ for all $\theta_1,\theta_2\in\mathbb{R}$. Similarly, 
  \begin{align*}
    X_{P'}=\left(S^1\times [0,1]\right)/\sim
  \end{align*}
  where $\left(e^{2\pi\sqrt{-1}\theta_1},0\right)\sim \left(e^{2\pi\sqrt{-1}\theta_2},0\right)$, $\left(e^{2\pi\sqrt{-1}\theta_1},1\right)\sim \left(e^{2\pi\sqrt{-1}\theta_2},1\right)$ for all $\theta_1,\theta_2\in\mathbb{R}$.
 
  The two  spaces are decipted in Figure \ref{space of A1}.
\end{exam}

\section{Root systems and Weyl polytopes}

We mainly follow \cite{bourbaki2008lie} and \cite{humphreys1992reflection} for facts of root systems.

Let $V$ be an $n$-dimensional real Euclidean space with an inner product $\left\langle-,-\right\rangle$, $R$ be a reduced crystallographic \textbf{root system} of rank $n$, $\Delta:=\{\alpha_1,\alpha_2,\cdots,\alpha_n\}$ be a fixed system of \textbf{simple roots}.
 The  vectors $\omega_1^\vee,\omega_2^\vee,\cdots,\omega_n^\vee$ satisfying $\left\langle\omega_i^\vee,\alpha_j\right\rangle=\delta_{ij}$ are called the \textbf{fundamental coweights}. The vectors $\alpha_1^\vee,\alpha_2^\vee,\cdots,\alpha_n^\vee$ satisfying $\alpha_i^\vee=\frac{2\alpha_i}{\left\langle\alpha_i,\alpha_i\right\rangle}$  are called \textbf{simple coroots}. The vectors $\omega_1,\omega_2,\cdots,\omega_n$ satisfying  $\left\langle\alpha_i^\vee,\omega_j\right\rangle=\delta_{ij}$ are called \textbf{fundamental weights}.

Each root $\alpha\in R$ can be written as a linear combination of simple roots with integral coefficients of the same sign. In this case, a root is called \textbf{positive} if all coefficients are non-negative, or \textbf{negative} if all coefficients are non-positive. Let $R^+$ be the set of all positive roots, $R^-$ be the set of all negative roots, then $R$ is the disjoint union of $R^+$ and $R^-$.

Each root $\alpha\in R$ determines a reflection $s_{\alpha}\in O(V)$ defined by $s_{\alpha}(x)=x-\left\langle x,\alpha\right\rangle\alpha^{\vee}$, and all these reflections $s_{\alpha}$'s generate a finite group $W$, called the \textbf{Weyl group}. Moreover, $W$ is generated by simple reflections $s_i:=s_{\alpha_i}$ for $i\in [n]:=\{1,2,\cdots,n\}$.

Let $I$ be a subset of $[n]$, one can consider the simple sub-system $\Delta_I:=\{\alpha_i\}_{i\in I}$, the root sub-system $R_I\subset R$ generated by $\Delta_I$, and $W_I$ the corresponding Weyl group. $W_I$ is called a \textbf{parabolic subgroup} of $W$.

Associated to each root $\alpha\in R$ there are affine hyperplanes $H_{\alpha,k}:=\{x\in V| \left\langle x,\alpha\right\rangle=k\}$ for $k\in\mathbb{Z}$. Each connected component of $V\setminus\mathop{\cup}\limits_{\alpha\in R,k\in\mathbb{N}}H_{\alpha,k}$ is called an \textbf{alcove}. Each connected component of $V\setminus\mathop{\cup}\limits_{\alpha\in R}H_{\alpha,0}$ is called an \textbf{chamber}.
 Pick out one certain alcove $A$ and one certain chamber $C$, given by
\begin{align*}
  A:&=\{x\in V|\, 0<\left\langle x,\alpha\right\rangle<1,\, \forall \alpha\in R^+\},\\
  C:&=\{x\in V|\, \left\langle x,\alpha_i\right\rangle> 0,\, \forall i\in [n]\},
\end{align*}
called the \textbf{fundamental alcove} and the \textbf{fundamental chamber} respectively. 
One can identify the closure $\bar{C}$ of the fundamental chamber $C$ with the quotient space $V/W$, which follows from the result below:

\begin{thm}[{\cite[Lemma 1.12 and Theorem 1.12]{humphreys1992reflection}}]\label{Wconjugate}\label{isotropy}
 $\bar{C}$ is the \textbf{fundamental domain for the $W$-action on $V$}, that is, each $y\in V$ lies in $W(x)$ for some unique $x\in \bar{C}$.
 And $x-y$ is a non-negative $\mathbb{R}$-linear combination of simple roots. 
 Furthermore, the isotropy group group of $x$, i.e. $\{s\in W:\,s(x)=x\}$, is generated by those simple reflections that fix $x$.
\end{thm}

\begin{exam}\label{exam A2}
  One can consider the root system $R$ of type $A_2$ lying in $V:=\{x=(x_1,x_2,x_3)\in\mathbb{R}^3:x_1+x_2+x_3=0\}$, with simple coroots $\alpha_1^{\vee}=(1,-1,0),\,\alpha_2^{\vee}=(0,1,-1)$
  and the Weyl group $W\cong S_3$, generated by reflections permuting coordinates of $\mathbb{R}^3$. 
  
  The data about $A_2$ is decipted in the following picture,
  where the black vectors indicate coroots, the red vectors indicate fundamental coweights, the dashed blue lines indicate $\{H_{\alpha,k}\}_{\alpha\in R, k\in\mathbb{N}}$, the gray area indicates $C$, the orange area indicates $A$. 
  \begin{figure}[H]
    \centering
    \begin{tikzpicture}[scale=0.6]
      \newdimen\R
      \R=3cm

      \fill[fill=gray,opacity=0.2] (0:0) -- (30:{1.2*\R}) -- (90:{1.2*\R}) -- cycle;
        \fill[fill=orange,opacity=0.6] (0:0) -- (30:{(1/3)*sqrt(3)*\R}) -- (90:{(1/3)*sqrt(3)*\R}) -- cycle;
      \foreach \x in {60,120,...,360} {\draw[thick,->] (0:0) -- (\x:\R);}
      \foreach \x in {30,90} {\draw[very thick,->,red] (0:0) -- (\x:{(1/3)*sqrt(3)*\R});}
      \node[right] at (30:{(1/3)*sqrt(3)*\R}) {{\color{red}$\omega_2^{\vee}$}};
      \node[left] at (90:{(1/3)*sqrt(3)*\R}) {{\color{red}$\omega_1^{\vee}$}};
      \foreach \x/\l/\p in
        { 60/{$\alpha_1^{\vee}+\alpha_2^{\vee}$}/right,
         120/{$\alpha_1^{\vee}$}/above,
         180/{$-\alpha_2^{\vee}$}/left,
         240/{$-\alpha_1^{\vee}-\alpha_2^{\vee}$}/left,
         300/{$-\alpha_1^{\vee}$}/below,
         360/{$\alpha_2^{\vee}$}/right
        }
        \node[label={\p:\l}] at (\x:\R) {};
        
        \foreach \x in {-60,0,60} {
          \foreach \y in {-3,-1.5,0,1.5,3} {
            \begin{scope}[shift=(\x:\y cm)]
              \draw[dashed,blue,opacity=0.8] (90+\x:2.5cm)--(90+\x:-2.5cm);
            \end{scope}
          }
        }
        
   \end{tikzpicture}
  \end{figure}
\end{exam}

\begin{thm}[{\cite[p.\,90]{humphreys1992reflection}}]\label{alcove simplex}
  When the root system $R$ is \textbf{irreducible}, i.e. not a disjoint union of proper subsets orthogonal to each other, $R$ has a \textbf{highest root} $\tilde{\alpha}=\sum_{i\in [n]}c_i\alpha_i$ with all $c_i$'s positive, the fundamental alcove $A$ is equal to
  $\{x\in V| \left\langle x,\tilde{\alpha}\right\rangle< 1,\,\left\langle x,\alpha_i\right\rangle> 0,\,\forall i\in [n]\}.$ Consequently the closure $\bar{A}$ of $A$ is an $n$-simplex with vertices $\frac{1}{c_i}\omega_i^{\vee}$ and $0$.
\end{thm}

Denote by $\mathcal{P},\mathcal{Q},\mathcal{P}^\vee,\mathcal{Q}^\vee$ the lattices generated by fundamental weights, simple roots, fundamental coweights, simple coroots respectively. 
One can see that $\mathcal{P}^\vee$ is dual to $\mathcal{Q}$, $\mathcal{P}$ is dual to $\mathcal{Q}^\vee$, $\mathcal{Q}\subset \mathcal{P}, \mathcal{Q}^\vee\subset \mathcal{P}^\vee$ and $[\mathcal{P}:\mathcal{Q}]=[\mathcal{P}^\vee:\mathcal{Q}^\vee]$. The Weyl group $W$ acts on $R$ by sending a root to another, consequently $W$ acts on $\mathcal{P},\mathcal{Q},\mathcal{P}^\vee$ and $\mathcal{Q}^\vee$ respectively.

$V$ acts on itself by translation. Hence $V$ admits actions of the \textbf{affine Weyl group} $\mathcal{Q}^{\vee}\rtimes W$ and the \textbf{extended affine Weyl goup} $\mathcal{P}^{\vee}\rtimes W$. 
More precisely, the action of $\mathcal{P}^{\vee}\rtimes W$ on $V$ is given by 
\begin{align*}
  \left(\mathcal{P}^{\vee}\rtimes W\right)\times V &\to V\\
  \left((p,w),v\right)&\mapsto p+w(v).
\end{align*}
As a subgroup of $\mathcal{P}^{\vee}\rtimes W$, $\mathcal{Q}^{\vee}\rtimes W$ inherits the action on $V$.

\begin{thm}[{\cite[p.\,186]{bourbaki2008lie}}]\label{quotient under Q}
  The closure $\bar{A}$ of $A$ is a fundamental domain for the action of $\mathcal{Q}^{\vee}\rtimes W$ on $V$, and hence
    the quotient space $\frac{V}{\mathcal{Q}^{\vee}\rtimes W}$ is homeomorphic to $\bar{A}$.
\end{thm}

 From the following lemma, one can get that $\mathcal{Q}^{\vee}\rtimes W$ is a normal subgroup of $\mathcal{P}^{\vee}\rtimes W$.

\begin{lem}\label{sublattice}
  If there are lattice inclusions $\mathcal{Q}^{\vee}\subset \mathcal{S}^{\vee}\subset \mathcal{P}^{\vee}$ and $\mathcal{S}^{\vee}$ is invariant under $W$-action, then $\mathcal{S}^{\vee}\rtimes W$ is a normal subgroup of $\mathcal{P}^{\vee}\rtimes W$.
\end{lem}
\begin{proof}
  Assume that $(q_1,w_1)\in \mathcal{S}^{\vee}\rtimes W$ and $(q_2,w_2)\in \mathcal{P}^{\vee}\rtimes W$, then 
  \begin{align*}
    (q_2,w_2)^{-1}(q_1,w_1)(q_2,w_2)=\left(w_2^{-1}\left(w_1(q_2)-q_2+q_1\right),w_2^{-1}w_1w_2\right).
  \end{align*}
  By definition, $w_1(q_2)-q_2$ lies in $\mathcal{Q}^{\vee}$ and hence $w_2^{-1}\left(w_1(q_2)-q_2+q_1\right)$ lies in $\mathcal{S}^{\vee}$.
\end{proof}

Note that $\mathcal{P}^{\vee}\rtimes W$ acts on the set $\{H_{\alpha,k}\}_{\alpha\in R,k\in\mathbb{N}}$, and hence on the set of alcoves. Let $G_A$ be the subgroup of $\mathcal{P}^{\vee}\rtimes W$
that maps $A$ to $A$ itself. 
By Theroem \ref{quotient under Q}, the affine Weyl group $\mathcal{Q}^{\vee}\rtimes W$ is simply transitive on the set of alcoves, one can see that 
$$\mathcal{P}^{\vee}\rtimes W=\left(\mathcal{Q}^{\vee}\rtimes W\right)\rtimes G_A$$
and then 
$$\frac{\mathcal{P}^{\vee}}{\mathcal{Q}^{\vee}}\cong \frac{\mathcal{P}^{\vee}\rtimes W}{\mathcal{Q}^{\vee}\rtimes W}\cong G_A,$$
from which we get the affine action of $\frac{\mathcal{P}^{\vee}}{\mathcal{Q}^{\vee}}$ on $\bar{A}$, and the homeomorphisms due to Theorem \ref{quotient under Q}
\begin{equation}\label{alcove quotient}
   \frac{V}{\mathcal{S}^{\vee}\rtimes W}\cong \frac{V}{\mathcal{Q}^{\vee}\rtimes W}\bigg/ \frac{\mathcal{S}^{\vee}\rtimes W}{\mathcal{Q}^{\vee}\rtimes W}  \cong \frac{\bar{A}}{\nicefrac{\mathcal{S}^{\vee}}{\mathcal{Q}^{\vee}}}.
\end{equation}
where $S^{\vee}$ is given in Lemma \ref{sublattice}.

\begin{rem}
  When $R$ is irreducible, the group $\frac{\mathcal{P}^{\vee}}{\mathcal{Q}^{\vee}}$ is finite, see \cite[p.\,189]{bourbaki2008lie}.
\end{rem}

Due to Theorem \ref{alcove simplex}, in general the closed fundamental alcove $\bar{A}$ is a product of simplices, and hence a polytope. The (local) contractibility of $\frac{\bar{A}}{\nicefrac{\mathcal{S}^{\vee}}{\mathcal{Q}^{\vee}}}$ is implied by the result below.

\begin{lem}\label{contractible}
  Let $P'\subset \mathbb{R}^{n}$ be a polytope of dimension $n$. There is a group $G$ acting on $\mathbb{R}^{n}$ as affine transformations such that $G$ acts on $P'$. Then the quotient $P'/G$ is contractible and locally contractible.
\end{lem}
\begin{proof}
Let $c$ be the center of mass of $P'$, formally $c$ is the  average of vertices of $P'$. Clearly $c$ is fixed by $G$.  
The map 
\begin{align*}
  H: P'\times [0,1]&\to P' \\
  \left(v,t\right)&\mapsto c+t(v-c)
\end{align*}
is $G$-equivariant with $G$ acting trivially on the unit interval $[0,1]$. $H$ is a deformation retraction of $P'$ onto $\{c\}$, it induces a homotopy 
$$H/G:\left(P'/G\right)\times [0,1]\to P'/G,$$ wihch is a deformation retraction of $P'/G$ onto $\{c\}$. Thus we get that $P'/G$ is contractible. 

Apply the slice theorem of general topology \cite[Theorem II.4.4]{bredon1972introduction}, any point $x\in P'$ has a $G_x$-invariant open neighborhhod $U\subset P'$ such that $G(U)/G\cong U/G_x$ is an open neighborhhod of $G(x)/G$ in $P'/G$. The open neighborhhod $U$ contains a neighborhhod $U'$ of $x$ that is the intersection of $P'$ and an open convex ball of $\mathbb{R}^{n}$. Then $\widetilde{U}:=\mathop{\cap}\limits_{g\in G_x}g\left(U'\right)$
is a $G_x$-invariant convex open neighborhhod  of $x$ in $P'$.
Analogously, the map
\begin{align*}
  \widetilde{U}\times [0,1]&\to \widetilde{U}\\
  (v,t)&\mapsto x+t(v-x) 
\end{align*}
induces a deformation retraction of $\widetilde{U}/G_x\cong G(\widetilde{U})/G$ onto $\{x\}$. Therefore $P'/G$ is locally contractible.
\end{proof}

\begin{rem}
  One can relate Lemma \ref{contractible} to Lie group theory.  Suppose that $G$ is a (connected compact) simple Lie group with a fixed maximal torus $T$ and a Weyl group $W$. Let $Inn(G)$ denote the inner automorphism group of $G$,
  one has homeomorphisms
  \begin{align*}
    G/Inn(G)\cong T/W \cong \frac{\bar{A}}{\mathcal{S}^{\vee}/\mathcal{Q}^{\vee}},
  \end{align*}
  where $\bar{A}, \mathcal{S}^{\vee},\mathcal{Q}^{\vee}$ are constructed according to $G$. The rightmost space is (locally) contractible due to Lemma \ref{contractible}, hence $G/Inn(G)$ and  $T/W$ are (locally) contractible. See \cite[p. 326]{bourbaki2005lie} for the homeomorphisms above.
  \end{rem}

Fix a point $a\in C$, the convex hull $P:={Conv}\left(W(a)\right)$ is a polytope, called the \textbf{Weyl polytope}. Say a face $F$ of the Weyl polytope $P$ is \textbf{dominant} if $\dim(F\cap C)=\dim F$. The following theorem shows face information of the polytope $P$.

\begin{thm}[{\cite[Proposition 3.2]{vinberg1991certain}}]\label{face of P}
  Every face of $P$ is $W$-conjugate to a unique dominant face. There is a one-to-one correspondence between the set of all dominant faces and the set of all subsets of $\Delta$, assigning to each subset $\Delta_I$ the face $F= {Conv}\left(W_I(a)\right)$ of dimension $|I|$.
\end{thm}

\begin{coro}\label{P is simple}
  $P$ is a simple polytope. For any $i\in [n]$, the affine hyperplane $\{x\in V: \left\langle x-a,\omega_i^{\vee}\right\rangle=0\}$ supports a dominant facet of $P$.
\end{coro}
\begin{proof}
  As $i$ ranges among $[n]$, ${Conv}(W_{[n]\setminus\{i\}}a)$ is a facet of $P$, and its normal vector is $-\omega_i^{\vee}$ since $a-s_k(a)=\left\langle a,\alpha^\vee_k\right\rangle\alpha_k$ for $k\in [n]\setminus\{i\}$.

  By Theorem \ref{face of P}, the vertex $a$ is contained in only $n$ facets ${Conv}(W_{[n]\setminus\{i\}}a)$ for $i\in [n]$. And by symmetry, this determines facets containing any other vertex.
\end{proof}

Let $F_i$ denote the dominant facet ${Conv}\left(W_{[n]\setminus\{i\}}(a)\right)$ for $i\in [n]$. One can see that other facets of $P$ are given by $s(F_i)$ with the normal vector $s(-\omega_i^{\vee})$ for $i\in [n]$ and $s\in W$.

\begin{coro}\label{ConjuOfDomiFace}
  \begin{itemize}
    \item[$(a)$] For any $i\ne j$ in $[n]$, $\{s(F_i)\}_{s\in W}\bigcap\{r(F_j)\}_{r\in W}=\emptyset$.
    \item[$(b)$] For any $(k,s)\in [n]\times W$, $F_k\cap s(F_k)$ is either $F_i$ or $\emptyset$.
    \item[$(c)$] Define $\Lambda:=\mathop{\bigcup}\limits_{I\subset  [n]}\{I\}\times\frac{W}{W_{[n]\setminus I}}$, then there is a one-to-one correspondence between $\Lambda$ and the set of all proper faces of $P$, assigning $(I,s)$ to $s\left(\cap_{i\in I}F_i\right)$, where one should take $\cap_{i\in \emptyset}F_i$ as $P$.
  \end{itemize}
\end{coro}
\begin{proof}
  $(a)$ is an easy consequence of Theorem \ref{face of P}.

 For $(b)$, if $F_k\cap s(F_k)\ne\emptyset$, then there is some $r\in W$ such that $r\left(F_k\cap s(F_k)\right)=rF_k\cap rsF_k$ is a dominant face,  the only possibile case is that $rF_k=rsF_k=F_k$ and then $F_k=sF_k$.

 For $(c)$, $W$ acts transitively on the set of  faces of $P$ by Theorem \ref{face of P}, and the isotropy group of each $\cap_{i\in I}F_i$ is $W_{[n]\setminus I}$. 
\end{proof}

Now we turn to $P/W$, identified with $P\cap \bar{C}$ by Theorem \ref{Wconjugate}. Write $H_i$ for the hyperplane $H_{\alpha_i,0}$.
In $P\cap \bar{C}$, there is only one vertex $a$ of $P$, then all $H_i$'s and $F_i$'s for $i\in [n]$ support all facets of $P/W$.

\begin{lem}\label{facets intersect}
  For each $i\in [n]$, $H_i\cap F_i=\emptyset$.
\end{lem}
\begin{proof}
  It suffices to show all vertices of $F_i$ are strictly on the same side of $H_i$. Due to Theorem \ref{Wconjugate}, any vertex $s(a)$ of $P\cap F_i$ for some $s\in W_{[n]\setminus\{i\}}$ satisfies that $s(a)-a=\sum_{k\in [n]\setminus\{i\}}c_k\alpha_k$ with all $c_k$'s non-positive.
  Then one can get that 
  $$\left\langle s(a),\alpha_i^\vee \right\rangle=\left\langle \sum_{k\in [n]\setminus\{i\}}c_k\alpha_k,\alpha_i^\vee\right\rangle+\left\langle a,\alpha_i^\vee\right\rangle=\sum_{k\in [n]\setminus\{i\}}c_k\left\langle \alpha_k,\alpha_i^\vee\right\rangle+\left\langle a,\alpha_i^\vee\right\rangle>0,$$
  since $\left\langle \alpha_k,\alpha_i^\vee\right\rangle\leq 0$ for all $k\in [n]\setminus\{i\}$.
\end{proof}

For subsets $I,J\subset [n]$, let $H_IF_{J}$ denote  $\left(\cap_{i\in I}H_i\right)\cap\left(\cap_{j\in J}F_j\right)$. One can say that $H_IF_{J}\cap\bar{C}$ gives out all faces of $P\cap\bar{C}$ as $I,J$ range among all disjoint subsets of $[n]$, which follows from the result below:
\begin{coro}\label{QuotientIsSimple}
  The quotient polytope $P\cap \bar{C}$ is simple. Furthermore, for any subset $I$ of $[n]$,  $H_IF_{[n]\setminus I}$ determines a vertex of $P\cap \bar{C}$. Consequently, the polytope $P\cap \bar{C}$ is combinatorially equivalent to the standard n-cube $[0,1]^n$.
\end{coro}
\begin{proof}
  It is clear that the vertex $0$ is only contained  in $H_i\cap\bar{C}$ for any $i\in [n]$, and the vertex $a$ is only contained in $F_j$ for any $j\in [n]$ by Corollary \ref{P is simple}. For any non-empty proper subset $I$ of $[n]$, $\cap_{i\in I}H_i$ passes through the center of mass of ${Conv}(W_{I}(a))$, which is a dominant face of $P$, as well as the intersection of dominant facets $F_j$ for $j\in [n]\setminus I$. Then $H_IF_{[n]\setminus I}$ is a vertex since the corresponding normal vectors $\{\alpha^{\vee}_i\}_{i\in I}\cup\{-\omega_j^{\vee}\}_{j\in [n]\setminus I}$ span $V$.

  Since $P\cap \bar{C}$ is a polytope of dimension $n$, every vertex of it must lie in the intersection of at least $n$ facets. More precisely, each vertex is contained in $H_IF_{[n]\setminus I}$ for some subset $I$ of $[n]$, and only those facets $H_i$'s, $F_j$'s contain the vertex due to Lemma \ref{facets intersect}.

  As for the combinatorially equivalent type, there is  a bijection between faces of $P\cap \bar{C}$ and those of $[0,1]^n$, which assigns to $H_IF_J\cap\bar{C}$ the product of $\{0\}$'s, $\{1\}$'s and $[0,1]$'s indexed by $I$, $J$ and $[n]\setminus\left(I\cup J\right)$ respectively.
\end{proof}
\begin{exam}\label{quotientPolyToA2}
  In the situation of Example \ref{exam A2}, one can choose $a$ to be $\alpha_1^{\vee}+\alpha_2^{\vee}$. See the picture below, the  Weyl polytope $P=Conv\left(W(a)\right)$ is indicated by the red area, and $P/W$ by the brown area.
  Clearly $P/W$ is combinatorially equivalent to $[0,1]^2$.
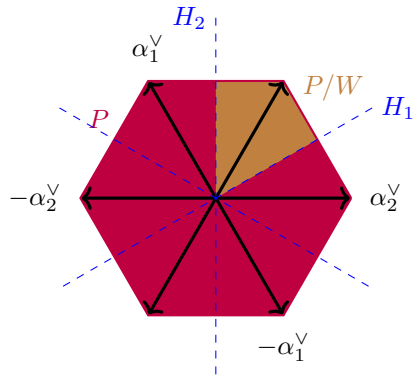
\begin{figure}[H]
  \centering
  \begin{tikzpicture}[scale=0.6]
    \newdimen\R
    \R=3cm
    \fill[purple,opacity=0.05] (0:\R)--(60:\R)--(120:\R)--(180:\R)--(240:\R)--(300:\R) -- cycle;
      \node[right,purple] at (150:3.5cm) {$P$};
      \fill[brown,opacity=0.8] (0:0)--(30:{0.5*sqrt(3)*\R})--(60:\R)--(90:{0.5*sqrt(3)*\R})--cycle;
      \node[right,brown] at (55:\R) {$P/W$};
    \foreach \x in {60,120,...,360} {\draw[very thick,->] (0:0) -- (\x:\R);}
    \draw[thick,purple] (360:\R) foreach \x in {60,120,...,360} { -- (\x:\R) };
    \foreach \x/\l/\p in
      {120/{$\alpha_1^{\vee}$}/above,
       180/{$-\alpha_2^{\vee}$}/left,
       300/{$-\alpha_1^{\vee}$}/below,
       360/{$\alpha_2^{\vee}$}/right
      }
      \node[label={\p:\l}] at (\x:\R) {};
      \node[left,blue] at (90:4cm) {$H_2$};
      \node[right,blue] at (30:4cm) {$H_1$};
      \foreach \x in {-60,0,60} {       
          \begin{scope}[shift=(\x: 0 cm)]
            \draw[dashed,blue] (90+\x:4cm)--(90+\x:-4cm);
          \end{scope}
      }
      
 \end{tikzpicture}
\end{figure}

\end{exam}

\section{The main result about the Weyl group}
In this section, suppose that $R$ is a reduced crystallographic root system of rank $n$, with a simple system $\Delta$. Let the lattice $N$ be the coweight lattice $\mathcal{P}^{\vee}$, its dual $M$ be the root lattice $\mathcal{Q}$. We consider $P=Conv(W(a))$ in $\mathcal{Q}\otimes_{\mathbb{Z}}\mathbb{R}=V$ with $a\in C\cap \mathcal{Q}$, and its normal vectors in $\mathcal{P}^{\vee}\otimes_{\mathbb{Z}}\mathbb{R}=V$. We write $\sigma$ for a cone of $\Sigma_P$, and $\sigma_{[n]}$ for a cone of $\Sigma_{P/W}$. 
Then for a point $p$ in $P$ (resp. $P\cap \bar{C}$), the cone $\sigma_p$ (resp. $\sigma_{[n],p}$) corresponds to the face of $P$ (resp. $P\cap \bar{C}$) whose relative interior contains $p$.

The Weyl polytope $P$ and its quotient polytope $P/W$ are associated to toric varieties $X_P$ and $X_{P/W}$ respectively.
\begin{lem}\label{toric orbifold}
  $X_P$ is a smooth toric variety, $X_{P/W}$ is a toric orbifold.
\end{lem}
\begin{proof}
  By the correspondence between toric varieties and rational fans, it suffices to show the normal fan $\Sigma_P$ is non-singular and $\Sigma_{P/W}$ is simplicial.

  The normal fans $\Sigma_P$ and $\Sigma_{P/W}$ are both simplicial since $P$ and $P/W$ are both simple by Corollary \ref{P is simple} and Corollary \ref{QuotientIsSimple}. The normal fan $\Sigma_P$ of $P$ contains maximal cones $Cone\left(-s\omega_1^{\vee},-s\omega_2^{\vee},\cdots,-s\omega_n^{\vee}\right)$ for all $s\in W$, and $s(\mathcal{P}^{\vee})=\mathcal{P}^{\vee}$. So the primitive generators of the rays of each maximal cone form a $\mathbb{Z}$-basis of $\mathcal{P}^{\vee}$. Thus $\Sigma_P$ is non-singular.
\end{proof}

\begin{rem}
  
  The normal fan $\Sigma_{P/W}$ contains maximal cones $Cone\left(\{\alpha_i^{\vee}\}_{i\in I}\cup\{-\omega_j^{\vee}\}_{j\in [n]\setminus I}\right)$ as $I$ ranges among subsets of $[n]$. The rays of $Cone\left(\{\alpha_i^{\vee}\}_{i\in [n]}\right)$ have generators $\{\alpha_i^{\vee}\}_{i\in [n]}$ (possibly not primitive), which form a $\mathbb{Z}$-basis of $\mathcal{Q}^{\vee}$. See Table \ref{indexofrank2} below for primitive generators of rays of $Cone\left(\{\alpha_i^{\vee}\}_{i\in [2]}\right)$ for the irreducible root system of rank 2.
   Therefore, $\Sigma_{P/W}$ is simplicial in general.
  \begin{table}[H]
    \begin{center}
      \caption{Primitive generators of rays of $Cone\left(\{\alpha_i^{\vee}\}_{i\in [2]}\right)$}
      \label{indexofrank2}
      \begin{tabular}{c|c|c}
        \textbf{Type} & \textbf{Generators} & \textbf{Primitive generators}\\
        \hline
        \multirow{2}{*}{$A_{2}$} & $\alpha_1^\vee=2\omega_1^{\vee}-\omega_2^{\vee}$ & $\alpha_1^\vee=2\omega_1^{\vee}-\omega_2^{\vee}$ \\
           & $\alpha_2^\vee=-\omega_1^{\vee}+2\omega_2^{\vee}$ & $\alpha_2^\vee=-\omega_1^{\vee}+2\omega_2^{\vee}$ \\
        \hline
        \multirow{2}{*}{$B_{2}$} & $\alpha_1^\vee=2\omega_1^{\vee}-\omega_2^{\vee}$  & $\alpha_1^\vee=2\omega_1^{\vee}-\omega_2^{\vee}$\\
         & $\alpha_2^\vee=-2\omega_1^{\vee}+2\omega_2^{\vee}$  &$\frac{1}{2}\alpha_2^\vee=-\omega_1^{\vee}+\omega_2^{\vee}$\\
        \hline
        \multirow{2}{*}{$G_{2}$} & $\alpha_1^\vee=2\omega_1^{\vee}-3\omega_2^{\vee}$ & $\alpha_1^\vee=2\omega_1^{\vee}-3\omega_2^{\vee}$  \\
         & $\alpha_2^\vee=-\omega_1^{\vee}+2\omega_2^{\vee}$ &$\alpha_2^\vee=-\omega_1^{\vee}+2\omega_2^{\vee}$ 
      \end{tabular}
    \end{center}
  \end{table}
  
\end{rem}

$X_P$ carries a natural action of $W$. Each element $s\in W$ induces a toric automorphism of the toric variety $X_P$, thus $W$ acts smoothly on the underlying smooth manifold of $X_P$.
The $s$ also  induces an automorphism of $P$, which consequently makes the diagram from Diagram \ref{natural model}
$$
\begin{tikzcd}
  \left(S_N\times P\right)/\sim\arrow[r,"\cong"]\arrow[d,"s"'] & X_P\arrow[d,"s"] \\
  \left(S_N\times P\right)/\sim\arrow[r,"\cong"'] & X_{P}
\end{tikzcd}
$$
commute. More precisely, suppose that $[t,p]\in X_P$, then $s$ acts on $X_P$ by sending $[t,p]$ to $[s(t),s(p)]$. 
Suppose that $\cap_{i\in I} F_i$ is the minimal face that contains $p$ for some unique subset $I\subset [n]$, then 
$$t\in \frac{S_N}{S_{N(\sigma_p)}}\cong \frac{V}{\mathcal{P}^{\vee}+Span(\{\omega_i^{\vee}:\,i\in I\})},$$ 
$s(p)$ lies in the relative interior of $\cap_{i\in I} s(F_i)$, and 
$$s(t)\in s\left(\frac{S_N}{S_{N(\sigma_p)}}\right)\cong s\left(\frac{V}{\mathcal{P}^{\vee}+Span(\{\omega_i^{\vee}:\,i\in I\})}\right)=\frac{V}{\mathcal{P}^{\vee}+Span(\{s(\omega_i^{\vee}):\,i\in I\})}\cong\frac{S_N}{S_{N(\sigma_{s(p)})}},$$
which shows that $W$-action on $X_P$ is well-defined.

\begin{lem}\label{equi.triangulation}
  $X_P/W$ is homeomorphic to  a simplicial complex, and so is $X_{P/W}$.
\end{lem}
\begin{proof}
  Due to Illman's results on triangulations of smooth equivariant manifolds \cite[Theorem 3.6]{illman1978smooth}, $X_P$ admits a smooth $W$-equivariant triangulation, which implies that $X_P/W$ admits an ordinary triangulation.

  $X_{P/W}$ is a toric variety and hence admits a triangulation, see \cite[Section 1]{hironaka1975triangulations} or \cite[p.\,1]{hofmann2009triangulation}.
\end{proof}

For any point $p\in P$, there exists $s\in W$ such that $s(p)\in P\cap\bar{C}$ and this $s(p)$ is unique (see Theorem \ref{Wconjugate}). In particular, when $p$ lies in $P\setminus \cup_{i\in [n]}H_i$, the $s$ is unique; when $p$ lies in $P\cap\left(\cap_{i\in I}H_i\right)$ for some maximal subset $I\subset [n]$, $W_I$  acts trivially on $p$ due to Theorem \ref{Wconjugate}, while the $W_I$-action on $S_N$ is not trivial. Consequently, we 
can construct a space as below:
$$\frac{S_N\times \left(P\cap \bar{C}\right)}{\sim_{ed}},$$
where $(t_1,p_1)\sim_{ed}(t_2,p_2)$ if and only if 
$p_1=p_2$ lying in the minimal face $H_IF_J\cap\bar{C}$ of $P\cap \bar{C}$ for some disjoint subsets $I,J\subset [n]$, and
$t_1$, $t_2$ represnet the same element of $\frac{S_N}{S_{N(\sigma_{p_1})}}\big/W_I$. Therefore we get the following result:

\begin{lem}\label{equi description}
  $X_P/W$ is homeomorphic to $\frac{S_N\times \left(P\cap \bar{C}\right)}{\sim_{ed}}$.
\end{lem}

Via the above description, the canonical quotient map $S_N\times \left(P\cap \bar{C}\right)\to X_{P/W}$ admits a factorization through $X_P/W$:
\begin{equation}\label{factorization}
  \begin{tikzcd}
    S_N\times \left(P\cap \bar{C}\right)\arrow[rr]\arrow[rd] &  &\frac{S_N\times \left(P\cap \bar{C}\right)}{\sim}  \\
     & \frac{S_N\times \left(P\cap \bar{C}\right)}{\sim_{ed}}\arrow[ru,"\Phi"'] &
  \end{tikzcd}
\end{equation}
Then we get a map $\Phi: X_P/W\to X_{P/W}$ from a quotient to a space underlying a toric variety, which is clearly continuous and surjective.

\begin{thm}\label{weak equivalence}
  The map $\Phi:X_P/W\to X_{P/W}$ is a homotopy equivalence.
\end{thm}

Since $X_P/W,\, X_{P/W}$ are simplicial complexes (Lemma \ref{equi.triangulation}), it suffices to show that $\Phi$ is a weak equivalence, which relies on a result of Stephen Smale as below. 
\begin{thm}[{\cite[p.\,1]{smale1957vietoris}}]\label{smale}
  Suppose $X, Y$ are connected, locally compact, separable metric spaces, $f: X\to Y$ is proper and onto, $X$ is locally $k$-connected. For each $y\in Y$, $f^{-1}(y)$ is locally $(k-1)$-connected and $(k-1)$-connected. Then $Y$ is locally $k$-connected and the induced homomorphism $f_{*}:\pi_{r}(X)\to \pi_{r}(Y)$ is isomorphic for $0\leq r\leq k-1$, onto for $r=k$.
\end{thm}

\begin{lem}\label{space prop}
  $X_P/W$, $X_{P/W}$ are connected, Hausdorff, compact, locally contractible, separable metric spaces.
\end{lem}
\begin{proof}
  By Lemma \ref{equi.triangulation}, $X_P/W,\, X_{P/W}$ are connected compact simplicial complexes. 
  A simplcial complex is locally contractible (or see \cite[Proposition A.4]{hatcher2002algebraic}). 
  A compact simplicial complex is finite (or see \cite[Proposition A.1]{hatcher2002algebraic}), then it can be embeded as a subcomplex of some simplex with the standard Euclidean metric, thus it is Hausdorff, locally compact and metrizable.
  One may need an fact that a compact metric space is separable. 
\end{proof}

\begin{rem}
  For Lemma \ref{space prop}, there is an elementary method to show that $X_{P}/W$ is locally contractible. 
  Suppose that $[t,p]\in X_{P}$ with $p$ lying in the relative interior of $H_IF_J\cap \bar{C}$ for some disjoint subsets $I,J\subset [n]$. One can choose appropriate open sets $U_t$ of the subtorus 
  $\frac{Span\left(\{\alpha_i^{\vee}\}_{i\in [n]\setminus J}\right)+\mathcal{P}^{\vee}}{\mathcal{P}^{\vee}}$, 
  $U_p$ of $P$ such that 
  $U_t\times S_{N(\sigma_p)}\times U_P\subset S_N\times P$ projects onto an open neighborhood of $[t,p]$ in $X_{P}$ that is also a slice at $[t,p]$ under the $W$-action (see \cite[Theorem II.4.4]{bredon1972introduction}). Then one can show that there is a $W_I$-equivariant deformation retraction from this slice onto a point, which implies the local contractibility of $X_P/W$.

\end{rem}

\begin{lem}\label{map prop}
  $\Phi$ is proper, that is, the preimage of a compact set is compact.
\end{lem}
\begin{proof}
  This follows from that $X_P/W$ and $X_{P/W}$ are Hausdorff and compact, and that $\Phi$ is continuous.
\end{proof}

The following work is to understand the fiber of $\Phi$. Before we go to the general result (Lemma \ref{preimagePhi}), here is an easy example:
\begin{exam}
One can consider the root system $R$ of type $A_1$ lying in $\mathbb{R}$ with a simple coroot $\alpha_1^{\vee}=2$ and a fundamental coweight $\omega_1^{\vee}=1$. The Weyl group $W$ is generated by $s_1$, the reflection of $\mathbb{R}$ around $0$. 
In Example \ref{eg:sapces of A1}, $P$ is a Weyl polytope and $P'$ is the quotient polytope $P/W$. W acts on $X_P$ by sending $\left[e^{2\pi\sqrt{-1}\theta},p\right]$ to $\left[e^{-2\pi\sqrt{-1}\theta},-p\right]$, then there is an homeomorphism
\begin{align*}
  X_P/W\cong \left(S^1\times [0,1]\right)/\sim_{ed}
  \end{align*}
  where $\left(e^{2\pi\sqrt{-1}\theta_1},0\right)\sim_{ed} \left(e^{-2\pi\sqrt{-1}\theta_1},0\right)$, $\left(e^{2\pi\sqrt{-1}\theta_1},1\right)\sim_{ed} \left(e^{2\pi\sqrt{-1}\theta_2},1\right)$ for all $\theta_1,\theta_2\in\mathbb{R}$. The map $\Phi:\left(S^1\times [0,1]\right)/\sim_{ed}\to \left(S^1\times [0,1]\right)/\sim$ is given by 
  $$\left[e^{2\pi\sqrt{-1}\theta},p\right]\mapsto \left[e^{2\pi\sqrt{-1}\theta},p\right].$$
  Then $\Phi^{-1}([t,p])$ is a point if $p\ne 0$, an interval (up to homeomorphism) if $p=0$.

 The three  spaces are decipted in Figure \ref{space of A1}.
\end{exam}

\begin{lem}\label{preimagePhi}
  Suppose that $[t,p]$ is a point of $X_{P/W}$, and $p$ lies in the relative interior of $H_IF_J\cap\bar{C}$ for some disjoint subsets $I,J\subset [n]$. Then $\Phi^{-1}([t,p])$ is homeomorphic to 
  $$\frac{S_{N(\sigma_{[n],p})}}{S_{N(\sigma_{p})}}\bigg/W_I\cong\frac{Span\left(\{\alpha_i^{\vee}\}_{i\in I}\right)}{Span\left(\{\alpha_i^{\vee}\}_{i\in I}\right)\bigcap\left(\mathcal{P}^{\vee}+Span\left(\{\omega_j^{\vee}\}_{j\in J}\right)\right)}\bigg/ W_I.$$
\end{lem}
In this lemma, recall that $\sigma_{[n],p}= Cone\left(\{\alpha_i^{\vee}\}_{i\in I}\cup\{-\omega_j^{\vee}\}_{j\in J}\right)$ and $\sigma_{p}= Cone\left(\{-\omega_j^{\vee}\}_{j\in J}\right)$. One can see that $W_I$ acts trivially on $p$ and on $S_{N(\sigma_{p})}$ by Theorem \ref{Wconjugate} and definition. This lemma shows, roughly speaking, that the fiber of $\Phi$ is homeomorphic to the quotient of a torus under the action of  the corresponding  Weyl group.

\begin{proof}[Proof of Lemma \ref{preimagePhi}]
  Here we use Diagram \ref{factorization} implicitly.
  The preimage of $[t,p]$ in $S_N\times \left(P\cap \bar{C}\right)$ is 
$$tS_{N(\sigma_{[n],p})}\times\{p\},$$
which is, via the projection into the first component, homeomorphic to 
$$tS_{N(\tilde{\sigma}_p)}=\frac{\tilde{t}+N+N(\tilde{\sigma}_p)\otimes_{\mathbb{Z}}\mathbb{R}}{N}=\frac{\tilde{t}+\mathcal{P}^\vee+Span\left(\{\alpha_i^{\vee}\}_{i\in I}\cup\{-\omega_j^{\vee}\}_{j\in J}\right)}{\mathcal{P}^\vee},$$
where $\tilde{t}\in V$ represents  $t\in S_{N}$.
Then the preimage of $[t,p]$ under $\Phi$ is homeomorphic to the image of $tS_{N(\tilde{\sigma}_p)}$ along the quotient map $S_N\to \frac{S_N}{S_{N(\sigma_{p})}}\big/W_I$, 
which is 
$$\frac{tS_{N(\sigma_{[n],p})}S_{N(\sigma_{p})}}{S_{N(\sigma_{p})}}\bigg/W_I=\frac{tS_{N(\sigma_{[n],p})}}{S_{N(\sigma_{p})}}\bigg/W_I\cong\frac{\tilde{t}+\mathcal{P}^{\vee}+Span\left(\{\alpha_i^{\vee}\}_{i\in I}\right)+Span\left(\{\omega_j^{\vee}\}_{j\in J}\right)}{\mathcal{P}^{\vee}+Span\left(\{\omega_j^{\vee}\}_{j\in J}\right)}\bigg/ W_I,$$
where the first homeomorphism follows from that $\sigma_{[n],p}$ contains ${\sigma}_p$.

Via the inner product, we can write $V$ as $Span\left(\{\alpha_i^{\vee}\}_{i\in I}\right)\bigoplus Span\left(\{\alpha_i^{\vee}\}_{i\in I}\right)^{\bot}$ with projections $\pi_1,\pi_2$ into each summand respectively. Then the translation by $-\pi_2\left(\tilde{t}\right)$ induces a homeomorphism
$$
\begin{tikzcd}
  \tilde{t}+\mathcal{P}^{\vee}+Span\left(\{\alpha_i^{\vee}\}_{i\in I}\right)+Span\left(\{\omega_j^{\vee}\}_{j\in J}\right)\arrow[d,"-\pi_2\left(\tilde{t}\right)"']\\
  \mathcal{P}^{\vee}+Span\left(\{\alpha_i^{\vee}\}_{i\in I}\right)+Span\left(\{\omega_j^{\vee}\}_{j\in J}\right),
\end{tikzcd}
$$
which is  equivariant under actions of $\mathcal{P}^{\vee}+Span\left(\{\omega_j^{\vee}\}_{j\in J}\right)$ and $W_I$,
and hence $\Phi^{-1}\left([t,p]\right)$ is homeomorphic to
\begin{align*}
    \frac{\mathcal{P}^{\vee}+Span\left(\{\alpha_i^{\vee}\}_{i\in I}\right)+Span\left(\{\omega_j^{\vee}\}_{j\in J}\right)}{\mathcal{P}^{\vee}+Span\left(\{\omega_j^{\vee}\}_{j\in J}\right)}\bigg/ W_I
   \cong  \frac{Span\left(\{\alpha_i^{\vee}\}_{i\in I}\right)}{Span\left(\{\alpha_i^{\vee}\}_{i\in I}\right)\bigcap\left(\mathcal{P}^{\vee}+Span\left(\{\omega_j^{\vee}\}_{j\in J}\right)\right)}\bigg/ W_I,
\end{align*}
where the left space is clearly homeomorphic to $\frac{S_{N(\sigma_{[n],p})}}{S_{N(\sigma_{p})}}\big/W_I$.
\end{proof}

 \begin{rem}
  Follow the notation of the above proof of Lemma \ref{preimagePhi}, one can see that $R_I=R\cap Span\left(\{\alpha_i^{\vee}\}_{i\in I}\right)$ is a root system in the Euclidean space $Span\left(\{\alpha_i^{\vee}\}_{i\in I}\right)$, with simple roots $\{\alpha_i\}_{i\in I}$ and simple coroots $\{\alpha_i^{\vee}\}_{i\in I}$. In this situation, $\pi_1\left(\mathcal{P}^{\vee}\right)$ is the coweight lattice: actually, any coweight $\tilde{\omega}^{\vee}$ of $R\cap Span\left(\{\alpha_i^{\vee}\}_{i\in I}\right)$ can be extended to $\omega^{\vee}\in V^*$ by vanishing on $\{\alpha_k\}_{k\in [n]\setminus I}$, then $\omega^{\vee} \in \mathcal{P}^{\vee}$ and $\pi_1(\omega^{\vee})=\tilde{\omega}^{\vee}$; the converse inclusion is easy.
 \end{rem}

Let $\mathcal{SP}_{I,J}^{\vee}$ denote $Span\left(\{\alpha_i^{\vee}\}_{i\in I}\right)\bigcap\left(\mathcal{P}^{\vee}+Span\left(\{\omega_j^{\vee}\}_{j\in J}\right)\right)$, $\mathcal{P}_I^{\vee}$ denote the coweight lattice of $R_I$, $\mathcal{Q}_I^{\vee}$ denote the coroot lattice of $R_I$, ${A_I}$ denote the fundamental alcove of $R_I$ lying in $Span\left(\{\alpha_i^{\vee}\}_{i\in I}\right)$. One can observe that
  there are  inclusions of $W_I$-invariant lattices
  $\mathcal{Q}_I^{\vee}\subset \mathcal{SP}_{I,J}^{\vee}\subset\mathcal{P}_I^{\vee}$, and hence $\mathcal{SP}_{I,J}^{\vee}\rtimes W_I$ is a normal subgroup of $\mathcal{P}_{I}^{\vee}\rtimes W_I$ due to Lemma \ref{sublattice}. 
From the relation \ref{alcove quotient}, one gets the homeomorphism
$$\frac{Span\left(\{\alpha_i^{\vee}\}_{i\in I}\right)}{\mathcal{SP}_{I,J}^{\vee}}\bigg/ W_I\cong\frac{\bar{A_I}}{\nicefrac{\mathcal{SP}_{I,J}^{\vee}}{\mathcal{Q}_I^\vee}}.$$

As $\bar{A_I}$ is a polytope in $Span\left(\{\alpha_i^{\vee}\}_{i\in I}\right)$, one gets the following result due to Lemma \ref{contractible}:

\begin{coro}\label{fiber contractible}
  The fiber of  $\Phi$ is  contractible and locally contractible.
\end{coro}

Lemma \ref{space prop}, \ref{map prop} and Corollary \ref{fiber contractible} together with Theorem \ref{smale} complete the proof of Theorem \ref{weak equivalence}.

\section{Generalization to parabolic subgroups}
We follow conditions in the beginning of Section 3.

Let $K$ be a subset of $[n]$, then simple reflections $s_k$'s for $k\in K$ generate $W_K$, a parabolic subgroup of $W$. Let $C_K$
be the fundamental chamber for the action of $W_K$ on V, i.e.,  
$$C_K=\{x\in V|\, \left\langle x,\alpha_k\right\rangle> 0,\, \forall k\in K\}.$$
One can follow the proof of Theorem \ref{Wconjugate} \cite[p.\,22]{humphreys1992reflection} to see this theorem also applies to the parabolic subgroup $W_K$, that is:
\begin{lem}\label{conjugate for paragp}
  $\overline{C_K}$ is a fundamental domain for the action of $W_K$ on $V$. Consequently, $P\cap\overline{C_K}$ is a fundamental domain for the action of $W_K$ on $P$.
\end{lem}

As a result, we can identify the quotient space $P/W_K$ with $P\cap\overline{C_K}$, while the latter one is a polytope.  Write $\sigma_K$ for a cone in the normal fan $\Sigma_{P/W_K}$. For a point $p$ in $P\cap\overline{C_K}$, the cone $\sigma_{K,p}$ corresponds to the face of $P\cap\overline{C_K}$ whose relative interior contains $p$. The polytope  $P/W_K$ is associated to a toric variety $X_{P/W_K}$.

Now we move to the face structure of $P\cap\overline{C_K}$. Note that there are two classes of facets of $P\cap\overline{C_K}$: 
\begin{itemize}
  \item for any $k\in K$, $H_k$  supports a facet $H_k\cap P$;
  \item define $\Lambda_K:=\left\{(\{i\},r)\in \Lambda:\,r(F_i)\cap\overline{C_K}\ne\emptyset\right\}$, then $r(F_i)$ supports a  facet $r(F_i)\cap\overline{C_K}$ for any $(\{i\},r)\in \Lambda_K$. 
\end{itemize}
The second class is based on Corollary \ref{ConjuOfDomiFace}, where $\Lambda=\mathop{\bigcup}\limits_{I\subset [n]}\{I\}\times\frac{W}{W_{[n]\setminus I}}$ gives out all faces  of $P$.

\begin{lem}\label{face of para quotient}
  The quotient polytope $P/W_K$, identified with $P\cap\overline{C_K}$, is simple. Moreover, each vertex is given by $s\left(H_IF_{[n]\setminus I}\right)$ for some $s\in W$ and some unique $I\subset [n]$ such that $\{s(H_i)\}_{i\in I}\subset \{H_k\}_{k\in K}$. Consequently, $X_{P/W_K}$ is a toric orbifold.
\end{lem}
\begin{proof}
  Suppose that $x\in V$ is a vertex of $P\cap\overline{C_K}$.
  Choose $K'\subset K$, $\Lambda_K'\subset\Lambda_K$ to be maximal in the sense that $x\in H_{k}\cap r(F_i)$ for all $k\in K'$, $\left(\{i\},r\right)\in\Lambda_K'$. Since $P\cap\overline{C_K}$ is a polytope of dimension $n$, one gets that $\left|K'\right|+\left|\Lambda_K'\right|\ge n$.
  By Theorem \ref{Wconjugate}, $x$ is $W$-conjugate to a unique point $y$ of $P\cap\bar{C}$. Suppose that $x=s(y)$ for some $s\in W$. By Corollary \ref{QuotientIsSimple}, the point $y$ is only supported by $\{H_i\}_{i\in I}$ and $\{F_j\}_{j\in J}$ for some disjoint subsets $I,J\subset [n]$. 
 Then the two bijections from Corollary \ref{ConjuOfDomiFace}
  \begin{align*}
    \{H_k\}_{k\in K'}\xrightarrow[\cong]{s^{-1}} \{H_i\}_{i\in I}\qquad\quad
    \{r(F_i)\}_{(i,r)\in\Lambda_K'}\xrightarrow[\cong]{s^{-1}} \{F_j\}_{j\in J}
  \end{align*}
  implies that $I\cup J=[n]$,
  $\left|K'\right|+\left|\Lambda_K'\right|=n$ and $x\in s\left(H_IF_{[n]\setminus I}\right)\cap\overline{C_K}$.
\end{proof}

Similarly one can show that a point $p$ in $P\cap\overline{C_K}$ must lie in the relative interior of $s\left(H_IF_{J}\right)\cap\overline{C_K}$ for some $s\in W$ and some unique disjoint subsets $I,J$ of $[n]$.
Then
the cone $\sigma_p$ is equal to $Cone\left(\{s(-\omega_j^{\vee})\}_{j\in J}\right)$ while
the cone $\sigma_{K,p}$ is equal to $Cone\left(\{s(\alpha_i^{\vee})\}_{i\in I}\cup\{s(-\omega_j^{\vee})\}_{j\in J}\right)$, invariant under the action of $sW_Is^{-1}$.
Note that $\{s(H_i)\}_{i\in I}$ is  a subset of $\{H_k\}_{k\in  K}$, and hence $\{s(\alpha_i^\vee)\}_{i\in I}$ is  a subset of $\{\alpha^\vee_k\}_{k\in  K}$, $sW_Is^{-1}$ is a subgroup of $W_K$.

\begin{exam}
  In Example \ref{quotientPolyToA2}, Let $K$ be $\{1\}\subset [2]$. Identify $P/W$ with $P\cap\bar{C}$, then $P/W$, $s_2(P/W)$ and $s_2s_1(P/W)$ constitute $P\cap\overline{C_K}$.
  See the picture below, where the gray area indicates $C_K$. 
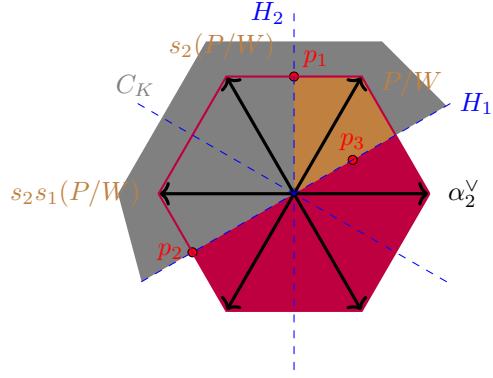
\begin{figure}[H]
  \centering
  \begin{tikzpicture}[scale=0.6]
    \newdimen\R
    \R=3cm
    \fill[purple,opacity=0.05] (0:\R)--(60:\R)--(120:\R)--(180:\R)--(240:\R)--(300:\R) -- cycle;
    \fill[gray,opacity=0.2] (30:{1.3*\R})--(60:{1.3*\R})--(120:{1.3*\R})--(180:{1.3*\R})--(210:{1.3*\R})--cycle;
      \node[right,gray] at (150:{1.6*\R}) {$C_K$};  

      \node[left,brown] at (180:3.2cm) {$s_2s_1(P/W)$};
      \fill[brown,opacity=0.8] (0:0)--(30:{0.5*sqrt(3)*\R})--(60:\R)--(90:{0.5*sqrt(3)*\R})--cycle;
      \node[right,brown] at (55:\R) {$P/W$};

      \node[above,brown] at (120:3.2cm) {$s_2(P/W)$};

      \draw[fill=red] (90:{0.5*sqrt(3)*\R}) circle[radius=0.1];
      \node[red, above right] at (90:{0.5*sqrt(3)*\R}) {$p_1$};

      \draw[fill=red] (210:{0.5*sqrt(3)*\R}) circle[radius=0.1];
      \node[red,left] at (210:{0.5*sqrt(3)*\R}) {$p_2$};

      \draw[fill=red] (30:{0.5*\R}) circle[radius=0.1];
      \node[red,above] at (30:{0.5*\R}) {$p_3$};
    \foreach \x in {60,120,...,360} {\draw[very thick,->] (0:0) -- (\x:\R);}
    \draw[thick,purple] (360:\R) foreach \x in {60,120,...,360} { -- (\x:\R) };
    \foreach \x/\l/\p in
      {       360/{$\alpha_2^{\vee}$}/right
      }
      \node[label={\p:\l}] at (\x:\R) {};
      \node[left,blue] at (90:4cm) {$H_2$};
      \node[right,blue] at (30:4cm) {$H_1$};
      \foreach \x in {-60,0,60} {       
          \begin{scope}[shift=(\x: 0 cm)]
            \draw[dashed,blue] (90+\x:4cm)--(90+\x:-4cm);
          \end{scope}
      }
      
 \end{tikzpicture}
\end{figure}
In this picture, we pick out three red points $p_1,p_2$ and $p_3$ in $P/W_K$, while $p_1$ and $p_3$ are also points of $P$ and $P/W$. When $p$ ranges among $\{p_1,p_2,p_3\}$, the ray generators of its associated cones $\sigma_p$, $\sigma_{[n],p}$ and $\sigma_{K,p}$ are listed in the following table:
\begin{table}[H]
  \begin{center}
    \begin{tabular}{c|c|c|c}
      \textbf{Point} & \textbf{Ray generators of $\sigma_p$} & \textbf{Ray generators of $\sigma_{[n],p}$} & \textbf{Ray generators of $\sigma_{K,p}$} \\
      \hline
      $p=p_1$ & $-\omega_1^\vee$ & $-\omega_1^\vee,\;\; \alpha_2^\vee$ & $-\omega_1^\vee$\\
      \hline
      \multirow{2}{*}{$p=p_2$} & \multirow{2}{*}{$s_2s_1(-\omega_1^\vee)=\omega_2^\vee$}  & \multirow{2}{*}{Undefined} & $s_2s_1(-\omega_1^\vee)=\omega_2^\vee$\\
       &   & & $s_2s_1(\alpha_2^\vee)=\alpha_1^\vee$\\
      \hline
      {$p=p_3$} & $0$ & $\alpha_1^\vee$ &$\alpha_1^\vee$  \\
    \end{tabular}
  \end{center}
\end{table}

\end{exam}

We can get similar results with those in Section 3, listed in the following lemma:

\begin{lem}\label{subgp poly}
  \begin{itemize}
    \item[$(a)$] There is a homeomorphism 
    $$X_P/W_K\cong \frac{S_N\times\left(P\cap\overline{C_K}\right)}{\sim_{ed}},$$
    where $(t_1,p_1)\sim_{ed}(t_2,p_2)$ if and only if 
    $p_1=p_2$ lying in the minimal face $s\left(H_IF_{J}\right)\cap\overline{C_K}$ of $P\cap\overline{C_K}$ for some $s\in W$ and some disjoint subsets $I,J\subset [n]$, and
    $t_1$, $t_2$ represnet the same element of $\frac{S_N}{S_{N(\sigma_{p_1})}}\big/sW_Is^{-1}$. 
    \item[$(b)$] $X_P/W_K$, $X_{P/W_K}$ are both finite simplicial complexes, and hence connected, Hausdorff, compact, locally contractible, separable metric spaces.
    \item[$(c)$] The quotient map $S_N\times\left(P\cap\overline{C_K}\right)\to \frac{S_N\times\left(P\cap\overline{C_K}\right)}{\sim}$ factors through a proper surjection
    $$\Phi_K:  \frac{S_N\times\left(P\cap\overline{C_K}\right)}{\sim_{ed}}\to \frac{S_N\times\left(P\cap\overline{C_K}\right)}{\sim}=X_{P/W_K}.$$
    \item[$(d)$] The fiber of  $\Phi_K$ is  contractible and locally contractible.
    More precisely, when $[t,p]$ is a point of $X_{P/W_K}$, and $p$ lies in the relative interior of some face $s\left(H_IF_{J}\right)\cap\overline{C_K}$, one can get that $\Phi_K^{-1}([t,p])$ is homeomorphic to 
     $$\frac{S_{N({\sigma}_{K,p})}}{S_{N(\sigma_{p})}}\bigg/sW_Is^{-1}\xrightarrow[\cong]{s^{-1}}\frac{Span\left(\{\alpha_i^{\vee}\}_{i\in I}\right)}{\mathcal{SP}_{I,J}^{\vee}}\bigg/ W_I,$$
     while the latter space is (locally) contractible.
  \end{itemize}
\end{lem}
\begin{proof}
  For $(a)$ see Lemma \ref{equi description}. For $(b)$ see Lemma \ref{equi.triangulation}, \ref{space prop}. For $(c)$ see Diagram \ref{factorization} and Lemma \ref{map prop}. For $(d)$ see Lemma \ref{preimagePhi} and Corollary \ref{fiber contractible}.
\end{proof}

Then apply Theorem \ref{smale}, we get 
\begin{thm}\label{result for parabolic}
  The map $\Phi_K:X_P/W_K\to X_{P/W_K}$ is a homotopy equivalence.
\end{thm}

\section{Results on alternative Weyl polytopes}

Suppose that in a real Euclidean space $V$ of dimension $n$, $R$ is a reduced crystallographic root system of rank $n$, with a simple system $\Delta$. Let the lattice $N$ be the coroot lattice $\mathcal{Q}^{\vee}$, its dual $M$ be the weight lattice $\mathcal{P}$. We consider the \textbf{alternative Weyl polytope} $Q:=Conv(W(a))$ in $\mathcal{P}\otimes_{\mathbb{Z}}\mathbb{R}=V$ with $a\in \bar{C}\cap\mathcal{P}$, and its normal vectors in $\mathcal{Q}^{\vee}\otimes_{\mathbb{Z}}\mathbb{R}=V$.

The quotient space $Q/W$ can be identified with a polytope $Q\cap\bar{C}$.
 We write $\tau$ for a cone of $\Sigma_{Q}$, and $\tau_{[n]}$ for a cone of $\Sigma_{Q/W}$. 
Then for a point $q$ in $Q$ (resp. $Q\cap \bar{C}$), the cone $\tau_q$ (resp. $\tau_{[n],q}$) corresponds to the face of $Q$ (resp. $Q\cap \bar{C}$) whose relative interior contains $q$.

\begin{rem}
  The dual results (Theorem \ref{dual res}, \ref{dual quo}) are very similar with the original ones (Theorem \ref{weak equivalence}, \ref{result for parabolic}). Actually, 
  the face structure of $Q$ and its quotient polytopes are the same as those of $P$ and its quotient polytopes, given in Theorem \ref{face of P}, Corollary \ref{ConjuOfDomiFace}, \ref{QuotientIsSimple} and Lemma \ref{face of para quotient}. The key difference is that $X_P$ is smooth while $X_Q$ is not, see Example \ref{dual ex of toric} below.
\end{rem}

\begin{exam}\label{dual ex of toric}
  Consider the alternative Weyl polytope $Q$ of Lie type $A_3$, there are  relations written in matrices:
  $$
  4\cdot\begin{bmatrix}
    \omega_1^{\vee}\\
    \omega_2^{\vee}\\
    \omega_3^{\vee}
  \end{bmatrix}=
  \begin{bmatrix}
    3 & 2 & 1\\
    2 & 4 & 2\\
    1 & 2 & 3
  \end{bmatrix}\cdot
  \begin{bmatrix}
    \alpha_1^{\vee}\\
    \alpha_2^{\vee}\\
    \alpha_3^{\vee}
  \end{bmatrix}
  $$
  This gives out primitive ray generators $-4\omega_1^\vee,-2\omega_2^\vee,-4\omega_3^\vee$ of the cone $\tau_a$ in $\Sigma_{Q}$, which doesn't form a basis of $\mathcal{Q}^{\vee}$, and hence $X_Q$ is an orbifold.
\end{exam}

\begin{lem}\label{QLocalContra}
  $X_{Q}/W$ and $X_{Q/W}$ are finite simplicial complexes.
\end{lem}
\begin{proof}
  Due to \cite[Theorem 1.1]{park1998semialgebraic}, the toric orbifold $X_{Q}$ is  (homeomorphic to) a $W$-equivariant simplicial complex, then $X_{Q}/W$ is a simplicial complex.
  See Lemma \ref{equi.triangulation} for triangulability of $X_{Q/W}$.
\end{proof}

We can get similar results with Lemma \ref{subgp poly} in the following lemma:
\begin{lem}\label{dual gp quo}
  \begin{itemize}
    \item[$(a)$] There is a homeomorphism 
    $$X_{Q}/W\cong \frac{S_N\times\left(Q\cap\bar{C}\right)}{\sim_{ed}},$$
    where $(t_1,q_1)\sim_{ed}(t_2,q_2)$ if and only if 
    $q_1=q_2$ lying in the minimal face $H_IF_{J}\cap\bar{C}$ of $Q\cap\bar{C}$ for some disjoint subsets $I,J\subset [n]$, and
    $t_1$, $t_2$ represnet the same element of $\frac{S_N}{S_{N(\tau_{q_1})}}\big/W_I$. 
    \item[$(b)$] $X_{Q}/W$, $X_{Q/W}$ are connected, Hausdorff, compact, locally contractible, separable metric spaces.
    \item[$(c)$] The canonical quotient map 
    $S_N\times\left(Q\cap\bar{C}\right)\to \frac{S_N\times\left(Q\cap\bar{C}\right)}{\sim}$ factors through a proper surjection
    $$\Psi:  \frac{S_N\times\left(Q\cap\bar{C}\right)}{\sim_{ed}}\to \frac{S_N\times\left(Q\cap\bar{C}\right)}{\sim}.$$
    \item[$(d)$] The fiber of  $\Psi$ is  contractible and locally contractible.
    More precisely, when $[t,q]$ is a point of $X_{Q/W}$, and $q$ lies in the relative interior of some face $H_IF_{J}\cap\bar{C}$, one can get that $\Psi^{-1}([t,q])$ is homeomorphic to 
     $$\frac{S_{N({\tau}_{[n],q})}}{S_{N(\tau_{q})}}\bigg/W_I\cong \frac{Span\left(\{\alpha_i^{\vee}\}_{i\in I}\right)}{Span\left(\{\alpha_i^{\vee}\}_{i\in I}\right)\bigcap\left(\mathcal{Q}^{\vee}+Span\left(\{\omega_j^{\vee}\}_{j\in J}\right)\right)}\bigg/ W_I,$$
     while the latter space is (locally) contractible.
  \end{itemize}
\end{lem}

Then apply Theorem \ref{smale}, we get a similar result with Theorem \ref{weak equivalence}:
\begin{thm}\label{dual res}
  The map $\Psi:X_{Q}/W\to X_{Q/W}$  is a homotopy equivalence.
\end{thm}

One can  also get a general result that is similar with Theorem \ref{result for parabolic}. Let $K$ be a subset of $[n]$, then the following result holds:

\begin{thm}\label{dual quo}
  There is a homotopy equivalence $\Psi_K:X_{Q}/W_K \to X_{Q/W_K}$ for any parabolic subgroup $W_K$ of $W$.
\end{thm}

\bibliographystyle{alpha}
\bibliography{ref}

\begin{thebibliography}{HMSS21}

\bibitem[Blu15]{blume2015toric}
Mark Blume.
\newblock Toric orbifolds associated to {Cartan} matrices.
\newblock In {\em Annales de l'Institut Fourier}, volume~65, pages 863--901, 2015.

\bibitem[Bou05]{bourbaki2005lie}
Nicolas Bourbaki.
\newblock Lie groups and lie algebras. chapters 7--9.
\newblock {\em Elements of Mathematics (Berlin)}, 2005.

\bibitem[Bou08]{bourbaki2008lie}
N.~Bourbaki.
\newblock {\em Lie Groups and Lie Algebras: Chapters 4-6}.
\newblock Elements de mathematique [series]. Springer Berlin Heidelberg, 2008.

\bibitem[Bre72]{bredon1972introduction}
Glen~E Bredon.
\newblock {\em Introduction to compact transformation groups}.
\newblock Academic press, 1972.

\bibitem[CLS11]{cox2011toric}
David~A Cox, John~B Little, and Henry~K Schenck.
\newblock {\em Toric varieties}, volume 124.
\newblock American Mathematical Soc., 2011.

\bibitem[Ful93]{fulton1993introduction}
William Fulton.
\newblock {\em Introduction to toric varieties}.
\newblock Number 131. Princeton university press, 1993.

\bibitem[Hat02]{hatcher2002algebraic}
Allen Hatcher.
\newblock {\em Algebraic Topology}.
\newblock Cambridge University Press, 2002.

\bibitem[Hir75]{hironaka1975triangulations}
Heisuke Hironaka.
\newblock Triangulations of algebraic sets.
\newblock In {\em Algebraic geometry (Proc. Sympos. Pure Math., Vol. 29, Humboldt State Univ., Arcata, Calif., 1974)}, volume~29, pages 165--185, 1975.

\bibitem[HMSS21]{horiguchi2021toric}
Tatsuya Horiguchi, Mikiya Masuda, John Shareshian, and Jongbaek Song.
\newblock Toric orbifolds associated with partitioned weight polytopes in classical types.
\newblock {\em arXiv preprint arXiv:2105.05453}, 2021.

\bibitem[Hof09]{hofmann2009triangulation}
Kyle~Roger Hofmann.
\newblock {\em Triangulation of Locally Semi-Algebraic Spaces.}
\newblock PhD thesis, The University of Michigan, 2009.

\bibitem[Hum92]{humphreys1992reflection}
James~E Humphreys.
\newblock {\em Reflection groups and Coxeter groups}.
\newblock Number~29. Cambridge university press, 1992.

\bibitem[Ill78]{illman1978smooth}
S{\"o}ren Illman.
\newblock Smooth equivariant triangulations of {G}-manifolds for {G} a finite group.
\newblock {\em Mathematische Annalen}, 233:199--220, 1978.

\bibitem[PS98]{park1998semialgebraic}
Dae~Heui Park and Dong~Youp Suh.
\newblock Semialgebraic ${G}$ ${CW}$ complex structure of semialgebraic ${G}$ spaces.
\newblock {\em Journal of the Korean Mathematical Society}, 35(2):371--386, 1998.

\bibitem[Sma57]{smale1957vietoris}
Stephen Smale.
\newblock A {Vietoris} mapping theorem for homotopy.
\newblock {\em Proceedings of the American mathematical society}, 8(3):604--610, 1957.

\bibitem[Vin91]{vinberg1991certain}
{\`E}~B Vinberg.
\newblock On certain commutative subalgebras of a universal enveloping algebra.
\newblock {\em Mathematics of the USSR-Izvestiya}, 36(1):1, 1991.

\end{thebibliography}

\end{document}